\documentclass[11pt, letterpaper]{article}
\usepackage[leqno]{amsmath}
\usepackage{abstract, amsthm, etoolbox, enumitem, hyperref, mathtools, sectsty, thmtools, tocloft}
\usepackage[all]{xy}

\usepackage[T1]{fontenc}
\usepackage[frenchstyle, largesmallcaps, light]{kpfonts}
\usepackage[mathcal]{eucal}
\usepackage{microtype}
\usepackage{geometry}
\allsectionsfont{\centering\fontseries{m}\scshape}
\sectionfont{\raggedright\fontseries{m}\scshape}

\allowdisplaybreaks

\newlength{\bibitemsep}
\newlength{\bibparskip}
\let\oldthebibliography\thebibliography
\renewcommand\thebibliography[1]{%
  \oldthebibliography{#1}%
  \setlength{\parskip}{\bibitemsep}%
  \setlength{\itemsep}{\bibparskip}%
}

\apptocmd{\thebibliography}{\raggedright}{}{}

\newcommand{\bs}{\boldsymbol}
\newcommand{\mb}{\mathbf}
\newcommand{\mc}{\mathcal}
\newcommand{\mf}{\mathfrak}
\newcommand{\mr}{\mathrm}

\declaretheoremstyle[spaceabove=6pt, spacebelo=6pt, headfont=\scshape, bodyfont=\itshape, postheadspace=.5em]{thm}
\declaretheoremstyle[spaceabove=6pt, spacebelo=6pt, headfont=\scshape, bodyfont=\normalfont, postheadspace=.5em]{defn}

\theoremstyle{thm}

\newtheorem{thm}{Theorem}
\numberwithin{thm}{section}
\newtheorem{cor}[thm]{Corollary}
\newtheorem{lemma}[thm]{Lemma}
\newtheorem{prop}[thm]{Proposition}

\theoremstyle{defn}

\newtheorem{defn}[thm]{Definition} 
\newtheorem*{motivation}{Motivation}

\newtheorem*{summary}{Summary}
\numberwithin{equation}{thm}

\setlist[enumerate,1]{label=\textup{(\textit{\roman*})},itemsep=.125\baselineskip,parsep=0pt,topsep=.125\baselineskip,listparindent=\parindent,itemindent=\parindent,wide=\parindent}
\setlist[enumerate,2]{label=\textup{(\alph*)},itemsep=.125\baselineskip,parsep=0pt,topsep=.125\baselineskip,listparindent=\parindent,leftmargin=3\parindent}
\setlist[enumerate,3]{itemsep=.125\baselineskip,parsep=0pt,topsep=.125\baselineskip,listparindent=\parindent,leftmargin=\parindent}
\setlist[itemize]{itemsep=.125\baselineskip,parsep=0pt,topsep=.125\baselineskip,listparindent=\parindent}

\usepackage{array, longtable}

\DeclarePairedDelimiter{\delimpar}{(}{)}
\DeclarePairedDelimiter{\delimbrk}{[}{]}
\DeclarePairedDelimiter{\delimbrc}{\{}{\}}
\DeclarePairedDelimiter{\delimbar}{\lvert}{\rvert}
\DeclarePairedDelimiter{\delimang}{\langle}{\rangle}

\newcommand{\prns}[2][0]{%
  \ifcase#1\relax
    \delimpar{#2}\or
    \delimpar[\big]{#2}\or
    \delimpar[\Big]{#2}\or
    \delimpar[\bigg]{#2}\or
    \delimpar[\Bigg]{#2}
  \else
    \delimpar*{#2}
  \fi}
\newcommand{\brk}[2][0]{%
  \ifcase#1\relax
    \delimbrk{#2}\or
    \delimbrk[\big]{#2}\or
    \delimbrk[\Big]{#2}\or
    \delimbrk[\bigg]{#2}\or
    \delimbrk[\Bigg]{#2}
  \else
    \delimbrk*{#2}
  \fi}
\newcommand{\brc}[2][0]{%
  \ifcase#1\relax
    \delimbrc{#2}\or
    \delimbrc[\big]{#2}\or
    \delimbrc[\Big]{#2}\or
    \delimbrc[\bigg]{#2}\or
    \delimbrc[\Bigg]{#2}
  \else
    \delimbrc*{#2}
  \fi}
\newcommand{\bars}[2][0]{%
  \ifcase#1\relax
    \delimbar{#2}\or
    \delimbar[\big]{#2}\or
    \delimbar[\Big]{#2}\or
    \delimbar[\bigg]{#2}\or
    \delimbar[\Bigg]{#2}
  \else
    \delimbar*{#2}
  \fi}
\newcommand{\ang}[2][0]{%
  \ifcase#1\relax
    \delimang{#2}\or
    \delimang[\big]{#2}\or
    \delimang[\Big]{#2}\or
    \delimang[\bigg]{#2}\or
    \delimang[\Bigg]{#2}
  \else
    \delimang*{#2}
  \fi}

\newcommand{\scr}[3][0]{\ifblank{#2}{\mathstrut{\vphantom{\prns[#1]{}}}_{#3}}{\ifblank{#3}{\mathstrut{\vphantom{\prns[#1]{}}}^{#2}}{\mathstrut{\vphantom{\prns[#1]{}}}^{#2}_{#3}}}}

\makeatletter
\def\scripts#1#2#3{\def\scripts@{\prns[#1]{#3}}\def\scripts@@{#2}\def\scripts@@@{#2}\@ifnextchar^\@sup\@nsup}
\def\@sup^#1{\def\scripts@@@{\scripts@@^{#1}}\@ifnextchar_\@supsub{\scripts@@@\scripts@}}
\def\@supsub_#1{\scripts@@@_{#1}\scripts@}
\def\@nsup{\@ifnextchar_{\@sub}{\scripts@@@\scripts@}}
\def\@sub_#1{\def\scripts@@@{\scripts@@_{#1}}\@ifnextchar^\@subsup{\scripts@@@\scripts@}}
\def\@subsup^#1{\scripts@@@^{#1}\scripts@}
\makeatother

\renewcommand{\1}[1]{\mathbf{1}_{#1}}
\newcommand{\ab}{\operatorname{\mathcal{A}b}}

\newcommand{\altdgnerve}[2][0]{\operatorname{\overline{\mathfrak{N}}}_{\mathrm{dg}}\ifblank{#2}{}{\prns[#1]{#2}}}

\newcommand{\calg}[2][0]{\scripts{#1}{\operatorname{\mathcal{CA}lg}}{#2}}
\newcommand{\cat}{\operatorname{\mathcal{C}at}}

\newcommand{\colim}{\operatorname*{colim}}

\newcommand{\cpx}[2][0]{\scripts{#1}{\operatorname{\mathcal{C}px}}{#2}}

\renewcommand{\d}{\operatorname{d}}
\renewcommand{\D}[2][0]{\scripts{#1}{\operatorname{\mathcal{D}}}{#2}}

\newcommand{\dgcat}{\operatorname{dg\mathcal{C}at}}

\newcommand{\dgnerve}[2][0]{\operatorname{\mathfrak{N}}_{\mathrm{dg}}\ifblank{#2}{}{\prns[#1]{#2}}}

\newcommand{\ext}[3][0]{\ifblank{#2}{\operatorname{ext}}{\scripts{#1}{\operatorname{ext}}{#2,\, #3}}}

\newcommand{\fin}{\operatorname{\mathcal{F}in}_*}
\newcommand{\fun}[3][0]{\ifblank{#2}{\operatorname{\mathcal{F}un}}{\scripts{#1}{\operatorname{\mathcal{F}un}}{#2,\, #3}}}

\newcommand{\h}{\mathfrak{h}}

\newcommand{\ho}[2][0]{\ifblank{#2}{\operatorname{ho}}{\scripts{#1}{\operatorname{ho}}{#2}}}

\renewcommand{\hom}[3][0]{\ifblank{#2}{\operatorname{hom}}{\scripts{#1}{\operatorname{hom}}{#2,\, #3}}}
\newcommand{\id}{\operatorname{id}}

\newcommand{\ind}[2][0]{\operatorname{\mathcal{I}nd}\ifblank{#2}{}{\prns[#1]{#2}}}

\newcommand{\inftyone}{\ensuremath{(\infty,1)}}

\newcommand{\K}[2][0]{\scripts{#1}{\operatorname{\mathcal{K}}}{#2}}
\newcommand{\KK}{\mathbf{K}}

\newcommand{\map}[3][0]{\ifblank{#2}{\operatorname{map}}{\scripts{#1}{\operatorname{map}}{#2,\, #3}}}

\newcommand{\mhs}{\operatorname{\mathcal{MHS}}}
\renewcommand{\mod}[2][0]{\ifblank{#2}{\operatorname{\mathcal{M}od}}{\scripts{#1}{\operatorname{\mathcal{M}od}}{#2}}}
\newcommand{\mor}[3][0]{\ifblank{#2}{\operatorname{mor}}{\scripts{#1}{\operatorname{mor}}{#2,\, #3}}}
\newcommand{\nerve}[2][0]{\operatorname{\mathfrak{N}}\ifblank{#2}{}{\prns[#1]{#2}}}

\newcommand{\on}{\operatorname}
\newcommand{\op}{^{\mathrm{op}}}

\newcommand{\qcat}{\operatorname{\mathcal{QC}at}}

\newcommand{\resp}[1]{\textup{(}resp.\@ #1\textup{)}}

\newcommand{\set}{\operatorname{\mathcal{S}et}}

\newcommand{\smod}[1]{\operatorname{s\mathcal{M}od}_{#1}\prns{\ab}}
\newcommand{\snerve}[2][0]{\operatorname{\mathfrak{N}_{\Delta}}\ifblank{#2}{}{\prns[#1]{#2}}}
\newcommand{\spc}[2][0]{\ifblank{#2}{\operatorname{\mathcal{S}pc}}{\scripts{#1}{\operatorname{\mathcal{S}pc}}{#2}}}

\newcommand{\sset}{\operatorname{\mathcal{S}et}_{\Delta}}

\newcommand{\trun}{\mathfrak{t}}

\title{\textsc{Verdier quotients of stable quasi-categories are localizations}}
\author{Brad Drew}
\date{}

\begin{document}
\maketitle

\begin{abstract}
\noindent
The Verdier quotient $\mc{T}/\mc{S}$ of a triangulated category $\mc{T}$ by a triangulated subcategory $\mc{S}$ is defined by a universal property with respect to triangulated functors out of $\mc{T}$.
However, $\mc{T}/\mc{S}$ is in fact a localization of $\mc{T}$, i.e., it is obtained from $\mc{T}$ by formally inverting a class of morphisms.
We establish the analogous result for small stable quasi-categories.
As an application, we explore the compatibility of Verdier quotients with symmetric monoidal structures.
In particular, we record a few useful elementary results on the quasi-categories associated with symmetric monoidal differential graded categories and derived categories of symmetric monoidal Abelian categories for which we were unable to locate proofs in the literature.
\end{abstract}

\tableofcontents

\subsection*{Introduction}

\noindent
In \cite[3.4]{Beilinson_absolute-hodge}, A.~Beilinson constructs a triangulated equivalence of the bounded derived category $\mr{D}^{\mr{b}}\prns{\mhs^{\mr{p}}_{\mb{Q}}}$ of rational polarizable mixed Hodge structures with the derived category $\mr{D}^{\mr{b}}_{\mc{H}^{\mr{p}}, \mb{Q}}$ of rational polarizable mixed Hodge complexes, and this equivalence has proved to be of fundamental importance in mixed Hodge theory.
In \cite{Drew_rectification-of-Deligne's}, we show that this triangulated equivalence can be enhanced to an equivalence of symmetric monoidal quasi-categories, thereby resolving several technical issues associated with the functoriality of triangulated categories.

In constructing this enhancement, however, we were unable to find proofs in the literature of some general results concerning the compatibility of Verdier quotients with symmetric monoidal structures and the passage from symmetric monoidal differential graded categories to symmetric monoidal quasi-categories. 
Specifically, $\mr{D}^{\mr{b}}_{\mc{H}}$ is the homotopy category of a Drinfeld quotient of a certain pretriangulated differential graded category $\bs{\mc{C}}$ by a full pretriangulated differential graded subcategory $\bs{\mc{A}} \subseteq \bs{\mc{C}}$, which carries a symmetric monoidal structure $\bs{\mc{C}}^{\otimes}$, which leads to the following questions:
\begin{itemize}
\item 
Does the symmetric monoidal structure $\bs{\mc{C}}^{\otimes}$ pass to the quasi-category $\dgnerve{\bs{\mc{C}}}$ associated with $\bs{\mc{C}}$?
\item
Is $\dgnerve{\bs{\mc{C}}}\scr{\otimes}{}$ a stable symmetric monoidal quasi-category?
\item
Does the Verdier quotient $\dgnerve{\bs{\mc{C}}}/\dgnerve{\bs{\mc{A}}}$ of $\dgnerve{\bs{\mc{C}}}$ by $\dgnerve{\bs{\mc{A}}}$ admit a more tractable description than as a cofiber in the quasi-category of small stable quasi-categories?
\item
Does $\dgnerve{\bs{\mc{C}}}/\dgnerve{\bs{\mc{A}}}$ inherit a natural symmetric monoidal structure?
\end{itemize} 
The answer in each case, as we prove below, is affirmative.
Of course, essentially none of this is specific to the differential graded category of mixed Hodge complexes and these questions are certain to arise in other contexts.
Given the general utility of these elementary results, they have been collected here in the hopes that others may find them and find uses for them.
 
\subsection*{Organization}

In \S1, we identify the Verdier quotient $\mc{C}/\mc{A}$ of a $\mf{U}$-small stable quasi-category $\mc{C}$ by a stable sub-quasi-category $\mc{A}$ with the localization of $\mc{C}$ with respect to the class $S$ of morphisms whose cofibers belong to $\mc{A}$ (\ref{thm.3}).
There are important properties easily established for $\mc{C}\brk{S^{-1}}$, such as its compatibility with well-behaved symmetric monoidal structures, that are otherwise difficult to prove directly for $\mc{C}/\mc{A}$;
it is in such cases that the utility of \ref{thm.3} manifests itself.
In \S2, we consider the particular case of quasi-categories arising from pretriangulated symmetric monoidal differential graded categories (\ref{dg.6}).

In \S3, we specialize further to the case of the derived category of a well-behaved symmetric monoidal Abelian category.
We also compare the construction of the bounded derived quasi-category $\D{\mb{A}}^{\mr{b}}$ of an Abelian category $\mb{A}$ as a Verdier quotient of the corresponding quasi-category of bounded cochain complexes up to cochain homotopy equivalence by the sub-quasi-category spanned by the acyclic complexes with that of the unbounded derived quasi-category $\D{\ind{\mb{A}}}$ of the category $\ind{\mb{A}}$ of ind-objects of $\mb{A}$ as the quasi-category corresponding to the injective model structure.
Specifically, we show that the former is a full sub-quasi-category of the latter when $\mb{A}$ is a Noetherian category (\ref{ab.5}).
In \S4, we develop \ref{ab.5} further under the assumption that $\mb{A}$ is also of finite homological dimension by characterizing the $\aleph_0$-presentable objects of $\D{\ind{\mb{A}}}$ and by showing that $\ind{\D{\mb{A}}^{\mr{b}}} \simeq \D{\ind{\mb{A}}}$ in this case.

\subsection*{Notation and conventions}

\noindent
\textsc{Grothendieck universes:}
We assume that each set is an element of a Grothendieck universe.
In particular, in the sequel, $\mf{U} \in \mf{V}$ denote fixed Grothendieck universes such that $\set$, $\ab$ and $\cat$, the categories of $\mf{U}$-sets, $\mf{U}$-small Abelian groups and $\mf{U}$-small categories, respectively, are $\mf{V}$-small.
All commutative rings and schemes will be $\mf{U}$-small unless context indicates otherwise.

\vspace{.25\baselineskip}

\noindent
\textsc{Quasi-categories:}
We freely employ the language of quasi-categories and higher algebra as developed in \cite{Lurie_higher-topos, Lurie_higher-algebra}.
We call \emph{qcategory} 
that which is elsewhere referred to variously as ``quasi-category'', ``quategory'', ``$\infty$-category'', ``\inftyone-category'' or ``weak Kan complex'' (\cite[1.1.2.4]{Lurie_higher-topos}).
While some of the avant-garde have already dubbed them simply ``categories'', a vestigial silent ``q'' will, I hope, strike a balance between ambiguity and polysyllabicism.

\vspace{.25\baselineskip}

\noindent
\textsc{Categories qua qcategories:}
We regard all categories as qcategories by tacitly taking their nerves.
As justification for this convention, observe that the nerve functor $\nerve{}: \cat \to \sset$ is right Quillen with respect to the model structure on $\cat$ whose weak equivalences and fibrations are the equivalences of categories and the isofibrations, respectively, and the Joyal model structure on the category $\sset$ of simplicial $\mf{U}$-sets, and the functor induced between the qcategories underlying these model structures is fully faithful (\cite[2.8]{Joyal_notes-on-quasi-categories}).

\vspace{.25\baselineskip}

\noindent
\textsc{Functors:}
We say that a functor $F: \mc{C} \to \mc{D}$ between qcategories is \emph{$\mf{U}$-continuous} \resp{\emph{$\mf{U}$-cocontinuous}} if it preserves $\mf{U}$-limits \resp{$\mf{U}$-colimits}, i.e., limits \resp{colimits} of $\mf{U}$-small diagrams.
We also write $F \dashv G$ to indicate that the functor $F: \mc{C} \to \mc{D}$ is left adjoint to the functor $G: \mc{D} \to \mc{C}$ (\cite[5.2.2.1]{Lurie_higher-topos}).

\vspace{.25\baselineskip}

\noindent
\textsc{Presentability:}
Let $\kappa$ denote an infinite regular $\mf{U}$-cardinal.
We preserve the terminology from the theory of $1$-categories (\cite{Adamek-Rosicky_locally-presentable}) and refer to an object $X$ of a qcategory $\mc{C}$ as \emph{$\kappa$-presentable} if $\map{X}{-}_{\mc{C}}: \mc{C} \to \spc{}$ preserves $\kappa$-filtered colimits, i.e., if it is ``$\kappa$-compact'' in the sense of \cite[5.3.4.5]{Lurie_higher-topos}.
We say that the qcategory $\mc{C}$ is \emph{locally $\mf{U}$-presentable} \resp{\emph{locally $\kappa$-presentable}} if it is ``presentable'' \resp{``$\kappa$-compactly generated''} in the sense of \cite[5.5.0.18, 5.5.7.1]{Lurie_higher-topos}.

\vspace{.25\baselineskip}

\noindent
\textsc{Localization:}
If $\mc{C}$ is a $\mf{U}$-small qcategory and $\mf{W}$ a class of morphisms of $\mc{C}$, then there exists a functor $\lambda: \mc{C}  \to \mc{C}\brk{\mf{W}^{-1}}$ with the universal property that, for each $\mf{U}$-small qcategory $\mc{D}$, composition with $\lambda$ induces a fully faithful functor $\fun{\mc{C}\brk{\mf{W}^{-1}}}{\mc{D}} \hookrightarrow \fun{\mc{C}}{\mc{D}}$ whose essential image $\fun{\mc{C}}{\mc{D}}_{\mf{W}}$ is spanned by those functors that send each element of $\mf{W}$ to an equivalence in $\mc{D}$ (\cite[1.3.4.2]{Lurie_higher-algebra}).
We refer to $\lambda$ or, abusively, $\mc{C}\brk{\mf{W}^{-1}}$, as a \emph{localization of $\mc{C}$ with respect to $\mf{W}$}.
If $\lambda$ admits a fully faithful right adjoint $\iota$, then we say that $\lambda$ is a \emph{reflective localization of $\mc{C}$}.
This applies in particular to the locally presentable setting:
if $\mc{C}$ is a locally $\mf{U}$-presentable category and $S$ is a $\mf{U}$-set of morphisms of $\mc{C}$, then the localization $\lambda: \mc{C} \to \mc{C} \brk{S^{-1}}$ is reflective, the essential image of its right adjoint is the full subqcategory spanned by the $S$-local objects, i.e., the objects $X \in \mc{C}$ such that, for each $f \in S$, the morphism $\map{f}{X}_{\mc{C}}$ is a weak homotopy equivalence, and $\mc{C}\brk{S^{-1}}$ is locally $\mf{U}$-presentable (\cite[5.5.4.15, 5.5.4.20]{Lurie_higher-topos}).

\vspace{.25\baselineskip}

\noindent
\textsc{Symmetric monoidal qcategories:}
A symmetric monoidal qcategory is a coCartesian fibration $p: \mc{C}^{\otimes} \to \fin$ such that the morphisms $\rho^i: \ang{n} \to \ang{1}$ given by 
$
\rho^i\prns{j}
:=
1$ if $i = j$ and 
$
\rho^i\prns{j} = *$ if
$i \neq j$
induce functors $\rho^i_!: \mc{C}^{\otimes}_{\ang{n}} \to \mc{C}^{\otimes}_{\ang{1}}$ which in turn induce equivalences $\mc{C}^{\otimes}_{\ang{n}} \simeq \prns{\mc{C}^{\otimes}_{\ang{1}}}^n$ (\cite[2.0.0.7]{Lurie_higher-algebra}).
We systematically suppress the fibration $p$ from the notation, referring to ``the symmetric monoidal qcategory $\mc{C}^{\otimes}$''.
We also refer to $\mc{C} := \mc{C}^{\otimes}_{\ang{1}}$ as the \emph{qcategory underlying $\mc{C}^{\otimes}$}.
Similarly, we use the notation $F^{\otimes}: \mc{C}^{\otimes} \to \mc{D}^{\otimes}$ for a possibly lax symmetric monoidal functor and $F: \mc{C} \to \mc{D}$ for the underlying functor.

Appealing to \cite[2.4.2.6]{Lurie_higher-algebra}, the category $\calg{\qcat^{\times}}$ of commutative algebra objects of $\qcat^{\times}$ is a convenient model for the qcategory of $\mf{U}$-small symmetric monoidal qcategories:
its objects correspond to $\mf{U}$-small symmetric monoidal qcategories and its morphisms to symmetric monoidal functors.
Similarly, by \cite[5.4.7]{Lurie_DAGVIII} and \cite[4.8.1.9]{Lurie_higher-algebra}, $\qcat^{\mr{Ex}}$ admits a symmetric monoidal structure $\qcat^{\mr{Ex}, \otimes}$ such that $\calg{\qcat^{\mr{Ex}, \otimes}}$ models the qcategory of $\mf{U}$-small stable qcategories.

\vspace{.25\baselineskip}

\noindent
\textsc{Qcategories underlying model categories:}
Essentially all the model categories appearing in the sequel will prove to be \emph{$\mf{U}$-combinatorial}, i.e., cofibrantly generated model categories with locally $\mf{U}$-presentable (\cite[1.8]{Beke_sheafifiable-homotopy}, \cite[\S A.2.6]{Lurie_higher-topos}, \cite[1.21]{Barwick_left-and-right}) underlying categories.

If $\mb{M}$ is a $\mf{U}$-combinatorial model category and $\mf{W}$ is its class of weak equivalences, then its \emph{underlying qcategory $\mb{M}\brk{\mf{W}^{-1}}$} is locally $\mf{U}$-presentable (\cite[1.3.4.15, 1.3.4.16]{Lurie_higher-algebra}).
If $\mb{M}^{\otimes}$ is a symmetric monoidal $\mf{U}$-combinatorial model category, then it admits an \emph{underlying locally $\mf{U}$-presentable symmetric monoidal qcategory $\mb{M}\brk{\mf{W}^{-}}\scr{\otimes}{}$} (\cite[4.1.3.6, 4.1.4.8]{Lurie_higher-algebra}).
Strictly speaking, the qcategory underlying the symmetric monoidal qcategory $\mb{M}\brk{\mf{W}^{-1}}\scr{\otimes}{}$ is defined to be the full subqcategory of $\mb{M}\brk{\mf{W}^{-1}}$ spanned by the cofibrant objects of $\mb{M}$, but the inclusion of this full subqcategory is an equivalence as each object of $\mb{M}$ is weakly equivalent to a cofibrant object.

\vspace{\baselineskip}

\noindent
\textsc{Notation:}
While we maintain most of the notations of \cite{Lurie_higher-topos,Lurie_higher-algebra}, the following list specifies the notable deviations therefrom and other frequently recurring symbols.

\renewcommand{\arraystretch}{1.2}

\begin{longtable}[c]{p{1.1in}>{\raggedright\arraybackslash}p{4.6in}}
$\mc{C}_{\aleph_0}$ & the full subqcategory of $\mc{C}$ spanned by the $\aleph_0$-presentable objects (\cite[5.3.4.5]{Lurie_higher-topos})
\\
$\mc{C}^{\times}$ & the Cartesian symmetric monoidal structure on the qcategory $\mc{C}$ (\cite[2.4.1.1]{Lurie_higher-algebra})
\\
$\mc{C}\brk{\mf{W}^{-1}}$ & a localization of the qcategory $\mc{C}$ with respect to the class of morphisms $\mf{W}$ (\cite[1.3.4.1]{Lurie_higher-algebra}) 
\\
$\ab$ & the category of $\mf{U}$-small Abelian groups
\\
$\calg{\mc{C}^{\otimes}}$ & the qcategory of commutative algebras in the symmetric monoidal qcategory $\mc{C}^{\otimes}$ (\cite[2.1.3.1]{Lurie_higher-algebra})
\\
$\fun{\mc{C}}{\mc{D}}_S$ & the full subqcategory of $\fun{\mc{C}}{\mc{D}}$ spanned by the functors that send each element of the set $S$ of morphisms of $\mc{C}$ an equivalence in $\mc{D}$
\\
$\fun{\mc{C}^{\otimes}}{\mc{D}^{\otimes}}^{\otimes}$ & the qcategory of symmetric monoidal functors $F^{\otimes}: \mc{C}^{\otimes} \to \mc{D}^{\otimes}$ (\cite[2.1.3.7]{Lurie_higher-algebra}) 
\\
$\h^r$ & the degree $r$ cohomology functor $\trun^{\leq r}\trun^{\geq r}: \mc{C} \to \mc{C}^{\heartsuit}$ of a t-structure on the stable qcategory $\mc{C}$ (\cite[1.2.1.4]{Lurie_higher-algebra})
\\
$\ho{\mc{C}}$ & the homotopy category of the qcategory $\mc{C}$ (\cite[1.2.3]{Lurie_higher-topos})
\\
$\ind{\mc{C}}$ & the ind-completion of the qcategory $\mc{C}$ (\cite[5.3.5.1]{Lurie_higher-topos})
\\
$\map{X}{Y}_{\mc{C}}$ & the mapping space between two objects $X$ and $Y$ of the qcategory $\mc{C}$ (\cite[1.2.2]{Lurie_higher-topos})
\\
$\mod{\mc{C}}_A$ & the qcategory of modules over $A \in \calg{\mc{C}^{\otimes}}$ (\cite[4.5.1.1]{Lurie_higher-algebra})
\\
$\nerve{\mc{C}}$ & the nerve of the category $\mc{C}$
\\
$\dgnerve{\bs{\mc{C}}}$, $\altdgnerve{\bs{\mc{C}}}$ & two constructions of the differential graded nerve of the differential graded category $\bs{\mc{C}}$ (\cite[1.3.1.6, 1.3.1.16]{Lurie_higher-algebra})
\\
$\snerve{\bs{\mc{C}}}$ & the simplicial nerve of the simplicial category $\bs{\mc{C}}$ (\cite[1.1.5.5]{Lurie_higher-topos})
\\
$\qcat$ & the qcategory of $\mf{U}$-small qcategories (\cite[3.0.0.1]{Lurie_higher-topos})
\\
$\qcat^{\mr{Ex}}$ & the qcategory of $\mf{U}$-small stable qcategories and exact functors (\cite[\S1.1.4]{Lurie_higher-algebra})
\\
$\sset$ & the category of simplicial $\mf{U}$-sets
\\
$\spc{}$ & the qcategory of $\mf{U}$-small spaces, i.e., the qcategory underlying the model structure on $\sset$ whose weak equivalences and fibrations are the weak homotopy equivalences and the Kan fibrations, respectively (\cite[1.2.16.1]{Lurie_higher-topos})
\\
$\trun^{\leq r}$, $\trun^{\geq r}$ & the truncations of a cohomological t-structure on a stable qcategory $\mc{C}$
\\
$\mf{U} \in \mf{V}$ & fixed Grothendieck universes
\end{longtable}

\renewcommand{\arraystretch}{1}

\section{Verdier quotients}

\begin{motivation}
If $\mc{C}^{\otimes}$ is a symmetric monoidal qcategory and $\mf{W}$ a class of morphisms of $\mc{C}$ such that $\id_X \otimes f \in \mf{W}$ for each $f \in \mf{W}$ and each $X \in \mc{C}$, then one can promote the localization $\lambda: \mc{C} \to \mc{C}\brk{\mf{W}^{-1}}$ to a symmetric monoidal functor (\cite[4.1.3.4]{Lurie_higher-algebra}).
If, however, $\mc{C}^{\otimes}$ is a \emph{stable} symmetric monoidal qcategory, it is perhaps more natural to ask if the Verdier quotient $q: \mc{C} \to \mc{C}/\mc{A}$ of $\mc{C}$ by a stable subqcategory $\mc{A}$ underlies a symmetric monoidal functor.
As we show in \ref{thm.3}, Verdier quotients are in fact localizations with respect to certain classes of morphisms.
Thus, the compatibility of a symmetric monoidal structure with a given Verdier quotient is equivalent to its compatibility with the corresponding localization.

The analogue of this useful result identifying Verdier quotients with localizations is well known in the setting of triangulated categories (\cite[Chapitre II, 2.2.11]{Verdier_des-categories}).
An analogue of this statement for certain Verdier quotients of stable, locally $\mf{U}$-presentable qcategories was established by A.~Blumberg, D.~Gepner and G.~Tabuada in \cite[5.7]{Blumberg-Gepner-Tabuada_universal-characterization}.
The proof of \ref{thm.3} consists of passing to the locally $\mf{U}$-presentable setting of \cite[5.7]{Blumberg-Gepner-Tabuada_universal-characterization} by taking ind-completions and verifying that this process is compatible with Yoneda embeddings.
\end{motivation}

\begin{summary}
Lemma \ref{thm.2} is a general result concerning the compatibility of ind-completions with localizations.
It allows us to prove the key result of this section, Theorem \ref{thm.3}, which in turn resolves the aforementioned problem of compatibility between Verdier quotients and symmetric monoidal structures (\ref{thm.4}).
\end{summary}

\begin{defn}
\label{thm.1}
If $\mc{C}$ is a $\mf{U}$-small stable qcategory and $\iota: \mc{A} \hookrightarrow \mc{C}$ the inclusion of a stable subqcategory, then the \emph{Verdier quotient $\mc{C}/\mc{A}$ of $\mc{C}$ by $\mc{A}$} is the cofiber $\on{cofib}\prns{\iota}$ in the qcategory $\qcat^{\mr{Ex}}$ of $\mf{U}$-small stable qcategories and exact functors.
\end{defn}

\begin{lemma}
\label{thm.2}
Let $\mc{C}$ be a $\mf{U}$-small qcategory, $S$ a set of morphisms of $\mc{C}$ and $T$ the image of $S$ under the Yoneda embedding $\mc{C} \hookrightarrow \ind{\mc{C}}$.
The canonical functor $\ind{\mc{C}\brk{S^{-1}}} \to \ind{\mc{C}}\brk{T^{-1}}$ is an equivalence.
\end{lemma}

\begin{proof}
If $\mc{A}$ and $\mc{B}$ are qcategories admitting $\aleph_0$-filtered $\mf{U}$-colimits, let $\fun{\mc{A}}{\mc{B}}^{\aleph_0} \subseteq \fun{\mc{A}}{\mc{B}}$ denote the full subqcategory spanned by the functors $F: \mc{A} \to \mc{B}$ preserving $\aleph_0$-filtered $\mf{U}$-colimits. 
If $\mc{A}$ and $\mc{B}$ are qcategories and $W$ is a class of morphisms of $\mc{A}$, we let $\fun{\mc{A}}{\mc{B}}_W \subseteq \fun{\mc{A}}{\mc{B}}$ denote the full subqcategory spanned by the functors $F: \mc{A} \to \mc{B}$ that send each element of $W$ to an equivalence of $\mc{B}$.
We also define $\fun{\mc{A}}{\mc{B}}^{\aleph_0}_W := \fun{\mc{A}}{\mc{B}}^{\aleph_0} \cap \fun{\mc{A}}{\mc{B}}_W$ when the latter is defined.
For each locally $\mf{U}$-presentable qcategory $\mc{D}$, we have natural equivalences
\begin{align*}
\fun{\ind{\mc{C}\brk{S^{-1}}}}{\mc{D}}^{\aleph_0}
&\simeq
\fun{\mc{C}\brk{S^{-1}}}{\mc{D}}
\simeq
\fun{\mc{C}}{\mc{D}}_S
\\
&\simeq
\fun{\ind{\mc{C}}}{\mc{D}}^{\aleph_0}_T
\simeq
\fun{\ind{\mc{C}}\brk{T^{-1}}}{\mc{D}}^{\aleph_0}.
\end{align*}
As $\ind{\mc{C}\brk{S^{-1}}}$ and $\ind{\mc{C}}\brk{T^{-1}}$ thus corepresent the same functor, they are equivalent.
\end{proof}

\begin{thm}
\label{thm.3}
Let $\iota: \mc{A} \hookrightarrow \mc{C}$ denote the inclusion of a stable subqcategory of the $\mf{U}$-small stable qcategory $\mc{C}$ and $S$ the $\mf{U}$-set of morphisms $f: X \to Y$ of $\mc{C}$ such that $\on{cofib}\prns{f} \in \mc{A}$.
The canonical functor $\phi: \mc{C}\brk{S^{-1}} \to \mc{C}/\mc{A}$ is an equivalence.
In particular, $\mc{C}\brk{S^{-1}}$ is stable.
\end{thm}

\begin{proof}
The canonical functor $\pi: \mc{C} \to \mc{C}/\mc{A}$ sends each element of $S$ to an equivalence:
a morphism of a stable qcategory is an equivalence if and only if its cofiber is a zero object.
Applying this observation and the universal property of the localization $\mc{C}\brk{S^{-1}}$, we obtain the functor $\phi$.
We have a homotopy commutative square
\[
\xymatrix{
\mc{C}\brk{S^{-1}}
\ar[r]_-{\phi}
\ar@{^{(}->}[d]
&
\mc{C}/\mc{A}
\ar@{^{(}->}[d]
\\
\ind{\mc{C}\brk{S^{-1}}}
\ar[r]^-{\ind{\phi}}_-{\sim}
&
\ind{\mc{C}/\mc{A}}
}
\]
in which the vertical arrows are the Yoneda embeddings and $\ind{\phi}$ is an equivalence by \ref{thm.2} and \cite[5.7, 5.13]{Blumberg-Gepner-Tabuada_universal-characterization}, so $\phi$ is fully faithful.
Since $\pi$ is essentially surjective and factors though $\phi$, $\phi$ is also essentially surjective.
\end{proof}

\begin{cor}
\label{thm.4}
Let $\mc{C}^{\otimes}$ be a $\mf{U}$-small stable symmetric monoidal qcategory, $\iota: \mc{A} \hookrightarrow \mc{C}$ a stable subqcategory such that $C \otimes A \in \mc{A}$ for each $C \in \mc{C}$ and each $A \in \mc{A}$, and $\pi: \mc{C} \to \mc{C}/\mc{A}$ the canonical functor.
Then $\pi$ underlies an exact symmetric monoidal functor $\pi^{\otimes}: \mc{C}^{\otimes} \to \prns{\mc{C}/\mc{A}}^{\otimes}$ such that, for each symmetric monoidal qcategory $\mc{D}$, the induced functor $\fun{\prns{\mc{C}/\mc{A}}^{\otimes}}{\mc{D}^{\otimes}}^{\otimes} \to \fun{\mc{C}^{\otimes}}{\mc{D}^{\otimes}}^{\otimes}$ is fully faithful and its essential image is spanned by the symmetric monoidal functors $F^{\otimes}$ sending each object of $\mc{A}$ to a zero object.
\end{cor}

\begin{proof}
With $S$ as in \ref{thm.3}, $\prns{-} \otimes \prns{-}$ sends elements of $S$ to equivalences separately in each variable by hypothesis.
The claim follows from \ref{thm.3} and \cite[4.1.3.4]{Lurie_higher-algebra}.
\end{proof}

\begin{defn}
\label{thm.5}
With the notation and hypotheses of \ref{thm.4}, we refer to $\pi^{\otimes}: \mc{C}^{\otimes} \to \prns{\mc{C}/\mc{A}}^{\otimes}$ as a \emph{symmetric monoidal Verdier quotient of $\mc{C}^{\otimes}$ by $\mc{A}$}.
\end{defn}

\section{Symmetric monoidal differential graded categories}

\begin{motivation}
There is a natural notion of a symmetric monoidal differential graded category, to wit, a differential graded category $\bs{\mc{C}}$ equipped with a differential graded tensor product bifunctor $\prns{-} \otimes \prns{-}$ and unit, associativity and commutativity constraints with respect to which it becomes a commutative monoid $\bs{\mc{C}}^{\otimes}$ in the symmetric monoidal category $\dgcat^{\otimes}_{\Lambda}$ of differential graded categories (\ref{dg.1}$(vi)$).
As mentioned in \cite[1.9.1]{Beilinson-Vologodsky_dg-guide}, however, this structure is too rigid to be compatible with Drinfeld quotients of differential graded categories.
In this section, we present a solution to this difficulty by applying results of G.~Faonte (\cite{Faonte_simplicial-nerve}) and V.~Hinich (\cite{Hinich_rectification-of-algebras}) to construct a symmetric monoidal qcategory underlying $\bs{\mc{C}}^{\otimes}$ and applying \ref{thm.4} to conclude that, under fairly general assumptions, if $\bs{\mc{C}}$ is pretriangulated and $\bs{\mc{A}} \subseteq \bs{\mc{C}}$ a pretriangulated differential graded subcategory, then the corresponding Verdier quotient of its underlying qcategory inherits a symmetric monoidal structure from $\bs{\mc{C}}^{\otimes}$.
\end{motivation}

\begin{summary}
In \ref{dg.1} and \ref{dg.2}, we recall some terminology for differential graded categories and, following \cite{Faonte_simplicial-nerve}, remark that the differential graded nerve of a pretriangulated differential graded category is a stable qcategory.
In \ref{dg.4}, \ref{dg.5}, we develop an observation of \cite{Hinich_rectification-of-algebras} concerning the compatibility of the differential graded nerve with symmetric monoidal structures.
This allows us to apply \ref{thm.4} to construct the desired symmetric monoidal Verdier quotients under suitable hypotheses.
In \ref{dg.7}, we record for future use the observation that, if $\bs{\mc{C}}^{\otimes}$ is a symmetric monoidal differential graded category, then the localization functor $\mr{Z}^0\prns{\bs{\mc{C}}} \to \mr{H}^0\prns{\bs{\mc{C}}}$ (\ref{dg.1}$(ii)$,$(iii)$) can be lifted to a symmetric monoidal functor from $\mr{Z}^0\prns{\bs{\mc{C}}}\scr{\otimes}{}$ to $\altdgnerve{\bs{\mc{C}}}\scr{\otimes}{}$, the differential graded nerve of $\bs{\mc{C}}$ with the symmetric monoidal structure of \ref{dg.6}.
\end{summary}

\begin{defn}
\label{dg.1}
Let $\Lambda$ be a unital commutative ring.
\begin{enumerate}
\item
We let $\dgcat_{\Lambda}$ denote the category of \emph{$\mf{U}$-small $\Lambda$-linear differential graded categories} and $\Lambda$-linear differential graded functors, i.e., the category of $\mf{U}$-small categories enriched in $\cpx{\mod{\ab}_{\Lambda}}$.
Given an object $\bs{\mc{C}} \in \dgcat_{\Lambda}$, we let $\hom{-}{-}^{\bullet}_{\bs{\mc{C}}} : \bs{\mc{C}}\op \times \bs{\mc{C}} \to \cpx{\mod{\ab}_{\Lambda}}$ denote $\cpx{\mod{\ab}_{\Lambda}}\scr{\otimes}{}$-enriched $\hom{}{}$-object bifunctor and, for each $n \in \mb{Z}$, we let $\hom{-}{-}^n_{\bs{\mc{C}}}: \bs{\mc{C}}\op \times \bs{\mc{C}} \to \mod{\ab}_{\Lambda}$ denote the component of $\hom{-}{-}^{\bullet}_{\bs{\mc{C}}}$ of degree $n$.
\item
The \emph{underlying category} $\mr{Z}^0\prns{\bs{\mc{C}}}$ of $\bs{\mc{C}} \in \dgcat_{\Lambda}$ is the category with the same objects as $\bs{\mc{C}}$ and morphisms given by $\mor{C}{C'}_{\mr{Z}^0\prns{\bs{\mc{C}}}} := \ker\prns{\hom{C}{C'}^{0}_{\bs{\mc{C}}} \to \hom{C}{C'}^1_{\bs{\mc{C}}}}$, i.e., the closed morphisms of degree zero.
Note that the category $\mr{Z}^0\prns{\bs{\mc{C}}}$ carries a natural $\Lambda$-linear, i.e., $\mod{\ab}_{\Lambda}\scr{\otimes}{}$-enriched, structure. 
\item
The \emph{homotopy category} $\mr{H}^0\prns{\bs{\mc{C}}}$ of $\bs{\mc{C}} \in \dgcat_{\Lambda}$ is the category with the same objects as $\bs{\mc{C}}$ and morphisms given by $\mor{C}{C'}_{\mr{H}^0\prns{\bs{\mc{C}}}} := \h^0\hom{C}{C'}^{\bullet}_{\bs{\mc{C}}}$.
Each $\Lambda$-linear differential graded functor $F:\bs{\mc{C}} \to \bs{\mc{D}}$ induces a functor $\mr{H}^0F: \mr{H}^0\prns{\bs{\mc{C}}} \to \mr{H}^0\prns{\bs{\mc{D}}}$.
\item
A $\Lambda$-linear differential graded functor $F: \bs{\mc{C}} \to \bs{\mc{D}}$ is \emph{homotopically fully faithful} if it induces a quasi-isomorphism
$
\hom{C}{C'}^{\bullet}_{\bs{\mc{C}}} 
\xrightarrow{F}
\hom{FC}{FC'}^{\bullet}_{\bs{\mc{D}}}
$
for each $\prns{C,C'} \in \bs{\mc{C}}^2$.
\item
A $\Lambda$-linear differential graded functor $F:  \bs{\mc{C}} \to \bs{\mc{D}}$ is a \emph{Dwyer-Kan equivalence} if it is homotopically fully faithful and $\mr{H}^0F$ is essentially surjective.
We let $\mf{W}_{\mr{DK}}$ denote the class of Dwyer-Kan equivalences in $\dgcat_{\Lambda}$.
\item
The category $\dgcat_{\Lambda}$ inherits a symmetric monoidal structure, denoted by $\dgcat^{\otimes}_{\Lambda}$, from $\cpx{\mod{\ab}_{\Lambda}}\scr{\otimes}{}$ in which the tensor product $\bs{\mc{C}} \otimes \bs{\mc{D}}$ of two objects is the $\Lambda$-linear differential graded category whose class of objects is the product $\bs{\mc{C}} \times \bs{\mc{D}}$ with 
\[
\hom{\prns{C, D}}{\prns{C', D'}}^{\bullet}_{\bs{\mc{C}} \otimes \bs{\mc{D}}} 
:= 
\hom{C}{C'}^{\bullet}_{\bs{\mc{C}}} \otimes \hom{D}{D'}^{\bullet}_{\bs{\mc{D}}}
\]
and with composition given by $\prns{g \otimes g'}\prns{f \otimes f'} := \prns{-1}^{\deg(g')\deg(f)}\prns{gf} \otimes \prns{g'f'}$.
A \emph{$\Lambda$-linear symmetric monoidal differential graded category} \resp{\emph{functor}} is an object \resp{morphism} of $\calg{\dgcat^{\otimes}_{\Lambda}}$.
\item
As explained in \cite[1.3.1.6, 1.3.1.11, 1.3.1.20]{Lurie_higher-algebra}, there is a \emph{differential graded nerve} functor $\dgnerve{}: \dgcat_{\Lambda}\brk{\mf{W}^{-1}_{\mr{DK}}} \to \qcat$ which assigns to each $\Lambda$-linear differential graded category $\bs{\mc{C}}$ a qcategory $\dgnerve{\bs{\mc{C}}}$ such that, in particular, $\mr{H}^0\prns{\bs{\mc{C}}} \simeq \ho{\dgnerve{\bs{\mc{C}}}}$.

We now recall a useful variant of $\dgnerve{}$, which we denote by $\altdgnerve{}: \dgcat_{\Lambda}\brk{\mf{W}^{-1}_{\mr{DK}}} \to \qcat$ and which will allow us to obtain stable symmetric monoidal qcategories from $\Lambda$-linear pretriangulated symmetric monoidal differential graded categories.
This functor $\altdgnerve{}$ is implicit in \cite[1.3.1.16]{Lurie_higher-algebra} and also studied in \cite[\S3.2]{Hinich_rectification-of-algebras} and \cite{Faonte_simplicial-nerve}.
By \cite[1.3.1.17]{Lurie_higher-algebra}, for each $\bs{\mc{C}} \in \dgcat_{\Lambda}$, there is an equivalence $\altdgnerve{\bs{\mc{C}}} \to \dgnerve{\bs{\mc{C}}}$ and hence $\mr{H}^0\prns{\bs{\mc{C}}} \simeq \ho{\altdgnerve{\bs{\mc{C}}}}$.

We have a sequence of lax symmetric monoidal right adjoint functors
\[
\cpx{\mod{\ab}_{\Lambda}}\scr{\otimes}{}
\xrightarrow{\trun^{\leq0, \otimes}}
\cpx{\mod{\ab}_{\Lambda}}^{\leq 0}\scr{\otimes}{}
\xrightarrow{N^{\otimes}}
\smod{\Lambda}^{\otimes}
\xrightarrow{\upsilon^{\otimes}}
\sset^{\times},
\]
where $\trun^{\leq0}$ is the truncation functor sending a cochain complex $M$ to the cochain complex
\[
\cdots
\to
M^{-2}
\xrightarrow{\d^{-2}}
M^{-1}
\xrightarrow{\d^{-1}}
\ker\prns{\d^0}
\to
0
\to
0
\to
\cdots,
\]
$N$ is half of the Dold-Kan correspondence and $\upsilon$ is the forgetful functor.
The composite lax monoidal functor induces a functor from the category of $\mf{U}$-small $\cpx{\mod{\ab}_{\Lambda}}\scr{\otimes}{}$-enriched categories, i.e., $\dgcat_{\Lambda}$, to that of $\mf{U}$-small $\sset^{\times}$-enriched categories. 
Unwinding the construction, one sees that this functor sends Dwyer-Kan equivalences of differential graded categories to Dwyer-Kan equivalences of $\sset^{\times}$-enriched categories.
The $\sset^{\times}$-category obtained from a $\smod{\Lambda}\scr{\otimes}{}$-category is necessarily fibrant (\cite[1.1.5.5]{Lurie_higher-topos}) and composing with the simplicial nerve functor $\snerve{}$ therefore provides a functor $\altdgnerve{}: \dgcat_{\Lambda}\brk{\mf{W}^{-1}_{\mr{DK}}} \to \qcat$.
\item
An object $\bs{\mc{C}} \in \dgcat_{\Lambda}$ is \emph{pretriangulated} if its image under the $\cpx{\mod{\ab}_{\Lambda}}\scr{\otimes}{}$-enriched Yoneda embedding is stable up to homotopy equivalence under translations and mapping cones of closed morphisms of degree zero.
We let $\dgcat^{\mr{pretr}}_{\Lambda} \subseteq \dgcat_{\Lambda}$ denote the full subcategory spanned by the pretriangulated differential graded categories and we abusively let $\dgcat^{\mr{pretr}}_{\Lambda}\brk{\mf{W}^{-1}_{\mr{DK}}}$ denote its localization with respect to $\dgcat^{\mr{pretr}}_{\Lambda} \cap \mf{W}_{\mr{DK}}$.

By \cite[Proposition 1]{Bondal-Kapranov_enhanced-triangulated}, if $\bs{\mc{C}}$ is pretriangulated, $\mr{H}^0\prns{\bs{\mc{C}}}$ admits a natural triangulated-category structure.
The distinguished triangles are given by the \emph{cofiber sequences}, i.e., the sequences of the form $X \xrightarrow{f} Y \to \on{cone}\prns{f} \to X\brk{1}$, with $f: X \to Y$ a closed morphism of degree zero.
\end{enumerate}
\end{defn}

\begin{thm}
[{\cite[4.3.1]{Faonte_simplicial-nerve}}]
\label{dg.2}
If $\KK$ is a field, then $\altdgnerve{}$ induces a functor $\dgcat^{\mr{pretr}}_{\KK}\brk{\mf{W}^{-1}_{\mr{DK}}} \to \qcat^{\mr{Ex}}$ and, for each $\bs{\mc{C}} \in \dgcat^{\mr{pretr}}_{\KK}$, the equivalence $\mr{H}^0\prns{\bs{\mc{C}}} \simeq \ho{\altdgnerve{\bs{\mc{C}}}}$ underlies an equivalence of triangulated categories.
\end{thm}

\begin{proof}
By \cite[4.3.1]{Faonte_simplicial-nerve}, for each $\bs{\mc{C}} \in \dgcat^{\mr{pretr}}_{\KK}$, the qcategory $\altdgnerve{\bs{\mc{C}}}$ is stable and $\mr{H}^0\prns{\bs{\mc{C}}}$ is $\ho{\altdgnerve{\bs{\mc{C}}}}$ are equivalent as triangulated categories.
To show that  $\altdgnerve{}: \dgcat^{\mr{pretr}}_{\KK} \to \qcat$ factors through $\qcat^{\mr{Ex}} \to \qcat$, it suffices to check that it sends each morphism $F: \bs{\mc{C}} \to \bs{\mc{D}}$ of $\dgcat^{\mr{pretr}}_{\KK}$ to an exact functor, i.e., that $\altdgnerve{F}$ preserves zero objects and cofiber sequences.
As explained in the proof of \cite[4.3.1]{Faonte_simplicial-nerve}, zero objects and cofiber sequences in $\bs{\mc{C}}$ \resp{$\bs{\mc{D}}$} induce zero objects and cofiber sequences in $\altdgnerve{\bs{\mc{C}}}$ \resp{$\altdgnerve{\bs{\mc{D}}}$}.
Since $F$ preserves zero objects and mapping cones, it follows that $\altdgnerve{F}$ preserves zero objects and cofiber sequences.
\end{proof}

\begin{lemma}
\label{dg.3}
Let $\KK$ be a field.
The projective model structure \textup{(\cite[2.3.11]{Hovey_model-categories})} on the category $\cpx{\mod{\ab}_{\KK}}$ equals the injective model structure \textup{(\cite[2.3.13]{Hovey_model-categories})}, which is therefore symmetric monoidal and satisfies the monoid axiom \textup{(\cite[3.3]{Schwede-Shipley_algebras-and-modules})}.
\end{lemma}

\begin{proof}
The weak equivalences and fibrations of the projective model structure are the quasi-isomorphisms and the epimorphisms, respectively.
By \cite[2.3.9]{Hovey_model-categories}, a morphism $f$ of $\cpx{\mod{\ab}_{\KK}}$ is a projective cofibration if and only if it is a degreewise split monomorphism with cofibrant cokernel.
Each monomorphism of $\KK$-modules splits as $\KK$ is a field.
Also, each $\KK$-module is projective and each acyclic complex of $\KK$-modules is contractible, so \cite[9.4]{Barthel-May-Riehl_six-model} shows that each object of $\cpx{\mod{\ab}_{\KK}}$ is cofibrant.
This shows that the projective cofibrations are precisely the injective cofibrations, which proves the first claim.
By \cite[4.2.13]{Hovey_model-categories}, the injective model structure is therefore symmetric monoidal.
The monoid axiom follows from the pushout-product axiom, since each object is cofibrant in the injective model structure (\cite[3.4]{Schwede-Shipley_algebras-and-modules}).
\end{proof}

\begin{thm}
[{\cite[\S3]{Hinich_rectification-of-algebras}}]
\label{dg.4}
Let $\KK$ be a field.
\begin{enumerate}
\item
The qcategory $\dgcat_{\KK}\brk{\mf{W}^{-1}_{\mr{DK}}}$ inherits a symmetric monoidal structure from $\dgcat^{\otimes}_{\KK}$.
\item
The functor $\altdgnerve{}$ underlies a lax symmetric monoidal functor $\altdgnerve{}^{\otimes}: \dgcat_{\KK}\brk{\mf{W}^{-1}_{\mr{DK}}}\scr{\otimes}{} \to \qcat^{\times}$.
\item
In particular, $\altdgnerve{}^{\otimes}$ induces a functor $\calg{\dgcat_{\KK}\brk{\mf{W}^{-1}_{\mr{DK}}}\scr{\otimes}{}} \to \calg{\qcat^{\times}}$.
\end{enumerate}
\end{thm}

\begin{proof}
By \ref{dg.3}, the enriched $\hom{}{}$-objects of each $\KK$-linear differential graded category are projectively cofibrant.
Following \cite[3.1.2]{Hinich_rectification-of-algebras}, we observe that tensor products preserve all Dwyer-Kan equivalences separately in each variable.
Assertion $(i)$ now follows from \cite[4.1.3.4]{Lurie_higher-algebra} and assertion $(ii)$ is explained in \cite[3.2.2]{Hinich_rectification-of-algebras}.
Assertion $(iii)$ follows from $(ii)$: lax symmetric monoidal functors induce functors between qcategories of commutative algebra objects.
\end{proof}

\begin{prop}
\label{dg.5}
Let $\KK$ be a field and $F^{\otimes}: \bs{\mc{C}}^{\otimes} \to \bs{\mc{D}}^{\otimes}$ a morphism of $\calg{\dgcat^{\otimes}_{\KK}}$ such that $\bs{\mc{C}}$ and $\bs{\mc{D}}$ are pretriangulated.
Then $\altdgnerve{F}^{\otimes}: \altdgnerve{\bs{\mc{C}}}^{\otimes} \to \altdgnerve{\bs{\mc{D}}}^{\otimes}$ is an exact symmetric monoidal functor between $\mf{U}$-small stable symmetric monoidal qcategories.
\end{prop}

\begin{proof}
By \ref{dg.4}, $F^{\otimes}$ induces a morphism $\altdgnerve{F}^{\otimes}$ of $\calg{\qcat^{\times}}$.
Under the given hypotheses, we claim that $\altdgnerve{F}^{\otimes}$ is in fact a morphism of $\calg{\qcat^{\mr{Ex}, \otimes}}$ (\cite[5.4.7]{Lurie_DAGVIII}, \cite[4.8.1.9]{Lurie_higher-algebra}), i.e., an exact symmetric monoidal functor between stable symmetric monoidal qcategories or, equivalently, that:
\begin{enumerate}
\item
$\altdgnerve{F}$ preserves finite colimits;
\item
for each $\langle n \rangle \in \fin$, the qcategories $\altdgnerve{\bs{\mc{C}}}^{\otimes}_{\langle n \rangle}$ and $\altdgnerve{\bs{\mc{D}}}^{\otimes}_{\langle n \rangle}$ are stable; and
\item
tensor products in $\altdgnerve{\bs{\mc{C}}}^{\otimes}$ and $\altdgnerve{\bs{\mc{D}}}^{\otimes}$ preserve finite colimits separately in each variable.
\end{enumerate}
By \ref{dg.2}, $\altdgnerve{F}$ is an exact functor between stable qcategories, which proves $(i)$.
Also, $\mf{U}$-products of stable qcategories are stable (\cite[1.1.4.3]{Lurie_higher-algebra}) and, by definition of a symmetric monoidal qcategory (\cite[2.0.0.7]{Lurie_higher-algebra}), $\altdgnerve{\bs{\mc{C}}}^{\otimes}_{\langle n \rangle}$ and $\altdgnerve{\bs{\mc{D}}}^{\otimes}_{\langle n \rangle}$ are both finite products of copies of $\altdgnerve{\bs{\mc{C}}}$ and $\altdgnerve{\bs{\mc{D}}}$, respectively, hence stable, proving $(ii)$.
For each object $X$ of $\bs{\mc{C}}$ \resp{$\bs{\mc{D}}$}, $X \otimes \prns{-}$ is a $\KK$-linear differential graded endofunctor of $\bs{\mc{C}}$ \resp{$\bs{\mc{D}}$}.
Unwinding the construction of $\altdgnerve{}^{\otimes}$, the endofunctor $X \otimes \prns{-}$ of $\altdgnerve{\bs{\mc{C}}}$ \resp{$\altdgnerve{\bs{\mc{D}}}$} is obtained from this differential graded endofunctor by applying $\altdgnerve{}$, so \ref{dg.2} implies that it too preserves finite colimits, proving $(iii)$.
\end{proof}

\begin{cor}
\label{dg.6}
Let $\KK$ be a field, $\bs{\mc{C}}^{\otimes} \in \calg{\dgcat^{\otimes}_{\KK}}$ such that $\bs{\mc{C}}$ is pretriangulated, and $\bs{\mc{A}} \subseteq \bs{\mc{C}}$ a pretriangulated full differential graded subcategory such that $C \otimes A \in \bs{\mc{A}}$ for each $C \in \bs{\mc{C}}$ and each $A \in \bs{\mc{A}}$.
The differential graded nerve $\altdgnerve{\bs{\mc{C}}}$ inherits a stable symmetric monoidal structure $\altdgnerve{\bs{\mc{C}}}\scr{\otimes}{}$ from $\bs{\mc{C}}^{\otimes}$ and the symmetric monoidal Verdier quotient \textup{(\ref{thm.5})} of $\altdgnerve{\bs{\mc{C}}}\scr{\otimes}{}$ by $\altdgnerve{\bs{\mc{A}}}$ exists.
\end{cor}

\begin{proof}
That $\altdgnerve{\bs{\mc{C}}}\scr{\otimes}{} \in \calg{\qcat^{\mr{Ex}, \otimes}}$ follows from \ref{dg.5} applied to the identity $\id_{\bs{\mc{C}}^{\otimes}}$.
By \ref{dg.2}, the inclusion $\bs{\mc{A}} \hookrightarrow \bs{\mc{C}}$ induces an exact functor $\altdgnerve{\bs{\mc{C}}} \to \altdgnerve{\bs{\mc{C}}}$, which is fully faithful as it is fully faithful on the corresponding homotopy categories (\cite[5.10]{Blumberg-Gepner-Tabuada_universal-characterization}).
The symmetric monoidal Verdier quotient therefore exists by \ref{thm.4}.
\end{proof}

\begin{prop}
\label{dg.7}
Let $\KK$ be a field, $\bs{\mc{C}}^{\otimes} \in \calg{\dgcat^{\otimes}_{\KK}}$. 
There exists a symmetric monoidal functor $\pi: \mr{Z}^0\prns{\bs{\mc{C}}}^{\otimes} \to \altdgnerve{\bs{\mc{C}}}^{\otimes}$ inducing the usual localization functor $\mr{Z}^0\prns{\bs{\mc{C}}} \to \mr{H}^0\prns{\bs{\mc{C}}}$ at the level of homotopy categories. 
\end{prop}

\begin{proof}
Let $\iota^0: \mod{\ab}_{\KK} \to \cpx{\mod{\ab}_{\KK}}$ denote the functor assigning to each $\KK$-module to the corresponding cochain complex concentrated in degree zero.
Then $\iota^0$ is left adjoint to the functor $\mr{Z}^0$ sending the cochain complex $K := \prns{K^{\bullet}, \d^{\bullet}}$ to the $\KK$-module $\mr{Z}^0\prns{K} := \ker\prns{\d^0: K^0 \to K^1}$.
We can also naturally promote $\iota^0$ to a symmetric monoidal functor $\iota^{0, \otimes}: \mod{\ab}_{\KK}\scr{\otimes}{} \to \cpx{\mod{\ab}_{\KK}}\scr{\otimes}{}$.

By \cite[4.3.1]{Cruttwell_normed-spaces}, $\iota^0 \dashv \mr{Z}^0$ induces a $2$-adjunction
$
\iota^0_*:
\mod{\ab}_{\KK}\scr{\otimes}{}\on{-}\cat
\rightleftarrows
\dgcat_{\KK}:
\mr{Z}^0_*
$
between the $2$-category of $\mf{U}$-small $\mod{\ab}_{\KK}\scr{\otimes}{}$-enriched categories and the $2$-category $\dgcat_{\KK}$.
Unwinding the definitions, for each $\KK$-linear category $\bs{\mc{A}}$, $\iota^0_*\bs{\mc{A}}$ has the same objects as $\bs{\mc{A}}$ and its $\hom{}{}$-objects are given by $\hom{A}{A'}^{\bullet}_{\iota^0_*\bs{\mc{A}}} := \iota^0\hom{A}{A'}_{\bs{\mc{A}}}$, where $\hom{A}{A'}_{\bs{\mc{A}}}$ denotes the $\KK$-module of morphisms from $A$ to $A'$ in $\bs{\mc{A}}$, and $\mr{Z}^0_*\bs{\mc{C}}$ is the $\KK$-linear category $\mr{Z}^0\prns{\bs{\mc{C}}}$ underlying $\bs{\mc{C}}$.
The counit of the adjunction $\iota^0_* \dashv \mr{Z}^0_*$ provides us with a $\KK$-linear differential graded functor $\iota^0_*\mr{Z}^0_*\bs{\mc{C}} = \iota^0_*\mr{Z}^0\prns{\bs{\mc{C}}} \to \bs{\mc{C}}$.

Note that $\mod{\ab}_{\KK}\scr{\otimes}{}\on{-}\cat$ admits a symmetric monoidal structure: for each pair of objects $\prns{\bs{\mc{A}}, \bs{\mc{B}}}$, their tensor product $\bs{\mc{A}} \otimes \bs{\mc{B}}$ is the $\KK$-linear category with objects given by pairs $\prns{A, B} \in \bs{\mc{A}} \times \bs{\mc{B}}$ and with $\hom{}{}$-objects given by $\hom{\prns{A,B}}{\prns{A',B'}}_{\bs{\mc{A}}\otimes \bs{\mc{B}}} := \hom{A}{A'}_{\bs{\mc{A}}} \otimes \hom{B}{B'}_{\bs{\mc{B}}}$.
With this definition, $\iota^0_*$ underlies a symmetric monoidal functor.
It follows that, for $\bs{\mc{C}}^{\otimes} \in \calg{\dgcat^{\otimes}_{\KK}}$, we may promote the $\KK$-linear differential graded functor $\iota^0_*\mr{Z}^0\prns{\bs{\mc{C}}} \to \bs{\mc{C}}$ to a symmetric monoidal $\KK$-linear differential graded functor.
By \ref{dg.4}$(iii)$, we have a symmetric monoidal functor $\altdgnerve{\iota^0_*\mr{Z}^0\prns{\bs{\mc{C}}}}^{\otimes} \to \altdgnerve{\bs{\mc{C}}}^{\otimes}$.
It remains to observe that, unwinding the construction of $\altdgnerve{}$ described in \ref{dg.1}$(vii)$, we have $\altdgnerve{\iota^0_*\mr{Z}^0\prns{\bs{\mc{C}}}}^{\otimes} = \mr{Z}^0\prns{\bs{\mc{C}}}^{\otimes}$.
\end{proof}

\section{Deriving Abelian categories}

\begin{motivation}
Let $\mb{A}^{\otimes}$ be a symmetric monoidal Abelian category whose tensor product is exact in each variable.
The bounded derived category $\mr{D}^{\mr{b}}\prns{\mb{A}}$ inherits a symmetric monoidal structure from $\mb{A}^{\otimes}$.
In this section, we construct an enhancement of this symmetric monoidal triangulated category, i.e., a stable symmetric monoidal qcategory $\D{\mb{A}}^{\mr{b}}\scr{\otimes}{}$ whose homotopy category recovers $\mr{D}^{\mr{b}}\prns{\mb{A}}\scr{\otimes}{}$ and which enjoys a suitable universal property in this regard.

There are natural situations in which the qcategory $\D{\mb{A}}^{\mr{b}}$ is simply too small for a given construction to work and one is obliged to embed it into a larger qcategory to avoid this obstacle.
For instance, the Adjoint Functor Theorem (\cite[5.5.2.9]{Lurie_higher-topos}) only applies to locally $\mf{U}$-presentable qcategories. 
We therefore also study the compatibility of the Verdier quotient $\D{\mb{A}}^{\mr{b}}$ with the unbounded derived category $\D{\ind{\mb{A}}}$ of its ind-completion.
Under a mild finiteness hypothesis on $\mb{A}$, the Yoneda embedding induces a fully faithful (symmetric monoidal) functor $\D{\mb{A}}^{\mr{b}}\scr{(\otimes)}{} \hookrightarrow \D{\ind{\mb{A}}}\scr{(\otimes)}{}$.
\end{motivation}

\begin{summary}
In Definition \ref{ab.1}, we define the bounded cochain and derived qcategories of $\mb{A}$.
Proposition \ref{ab.2} equips $\D{\mb{A}}^{\mr{b}}$ with a symmetric monoidal structure enjoying the correct universal property.
After recalling in Definition \ref{ab.3} the notions of Grothendieck Abelian categories and the injective model structure on unbounded complexes in such categories, we mention in Lemma \ref{ab.4} a sufficient condition for the unbounded derived qcategory of a Grothendieck Abelian category to inherit a symmetric monoidal structure.
Proposition \ref{ab.5} provides sufficient conditions for the (symmetric monoidal) bounded derived qcategory $\D{\mb{A}}^{\mr{b}}\scr{(\otimes)}{}$, constructed as a Verdier quotient, to be compatible with the (symmetric monoidal) unbounded derived category $\D{\ind{\mb{A}}}\scr{(\otimes)}{}$ of the ind-completion of $\mb{A}$, constructed as the (symmetric monoidal) qcategory underlying the injective model structure on $\cpx{\ind{\mb{A}}}\scr{(\otimes)}{}$.
\end{summary}

\begin{defn}
\label{ab.1}
Let $\mb{A}$ be a $\mf{U}$-small Abelian category.
\begin{enumerate}
\item
The category $\cpx{\mb{A}}^{\mr{b}}$ \resp{$\cpx{\mb{A}}$} of bounded \resp{unbounded} cochain complexes in $\mb{A}$ is a $\mf{U}$-small differential graded category.
Its differential graded nerve $\K{\mb{A}}^{\mr{b}} := \dgnerve{\cpx{\mb{A}}^{\mr{b}}}$ \resp{$\K{\mb{A}} := \dgnerve{\cpx{\mb{A}}}$}, as defined in \cite[1.3.1.6, 1.3.4.5]{Lurie_higher-algebra}, is a stable qcategory (\cite[1.3.2.10]{Lurie_higher-algebra}) and its homotopy category is the usual bounded \resp{unbounded} cochain homotopy category of $\mb{A}$ (\cite[1.3.1.11]{Lurie_higher-algebra}).
\item
Let $\iota: \on{\mc{A}c}\prns{\mb{A}} \hookrightarrow \K{\mb{A}}^{\mr{b}}$ denote the inclusion of the full subqcategory spanned by the acyclic complexes.
The \emph{bounded derived qcategory}  of $\mb{A}$ is the Verdier quotient $\D{\mb{A}}^{\mr{b}} := \K{\mb{A}}^{\mr{b}}/\on{\mc{A}c\prns{A}}$.
By \ref{thm.3}, the homotopy category of $\D{\mb{A}}^{\mr{b}}$ is the usual bounded derived category of $\mb{A}$.
\item
Suppose $\mb{A}^{\otimes}$ is a symmetric monoidal structure on $\mb{A}$.
Then $\cpx{\mb{A}}^{\mr{b}}$ \resp{$\cpx{\mb{A}}$} inherits a symmetric monoidal structure given informally by
\begin{equation*}
\label{ab.1.1}
\prns{K \otimes L}^n 
:= 
\bigoplus_{r \in \mb{Z}} \prns{K^r \otimes L^{n-r}},
\qquad
\d\prns{x \otimes y}
:=
\d\prns{x} \otimes y + \prns{-1}^{\deg\prns{x}}x \otimes \d\prns{y}.
\end{equation*}
By \cite[1.3.4.5, 4.1.3.4]{Lurie_higher-algebra}, there is a canonical symmetric monoidal structure $\K{\mb{A}}^{\mr{b}}\scr{\otimes}{}$ with respect to which the localization $\cpx{\mb{A}}^{\mr{b}}\scr{\otimes}{} \to \K{\mb{A}}^{\mr{b}}\scr{\otimes}{}$ is symmetric monoidal, and similarly for $\cpx{\mb{A}}\scr{\otimes}{} \to \K{\mb{A}}\scr{\otimes}{}$.
\item
In particular, in the setting of $(iii)$, if $\prns{-} \otimes \prns{-}$ is exact separately in each variable, $K$ and $L$ are two bounded-above cochain complexes in $\mb{A}$ and $L$ is acyclic, then $K \otimes L$, being the total complex of a double complex with acyclic columns and bounded diagonals, is acyclic.
\end{enumerate}
\end{defn}

\begin{prop}
\label{ab.2}
If $\mb{A}^{\otimes}$ is a $\mf{U}$-small symmetric monoidal Abelian category in which $\prns{-} \otimes \prns{-}$ is exact separately in each variable, then there exists a symmetric monoidal Verdier quotient $\D{\mb{A}}^{\mr{b}}\scr{\otimes}{}$ of $\K{\mb{A}}^{\mr{b}}\scr{\otimes}{}$ by $\on{\mc{A}c}\prns{\mb{A}}$.
\end{prop}

\begin{proof}
Observe that $\prns{-} \otimes \prns{-}$ preserves $\on{\mc{A}c}\prns{\mb{A}}$ separately in each variable and apply \ref{thm.4}.
\end{proof}

\begin{defn}
\label{ab.3}
Let $\mb{A}$ be a locally $\mf{U}$-small Abelian category.
\begin{enumerate}
\item
The category $\mb{A}$ is \emph{$\mf{U}$-Grothendieck Abelian} if it is an Abelian, locally $\mf{U}$-presentable category in which $\aleph_0$-filtered $\mf{U}$-colimits preserve finite limits.
By \cite[3.10]{Beke_sheafifiable-homotopy}, this is equivalent to the classical definition as an AB5 category with a generator.
If $\mb{A}$ is essentially $\mf{U}$-small, then its ind-completion $\ind{\mb{A}}$ is $\mf{U}$-Grothendieck Abelian.
\item
If $\mb{A}$ is $\mf{U}$-Grothendieck Abelian, then $\cpx{\mb{A}}$ admits a $\mf{U}$-combinatorial model structure (\cite[A.2.6.1]{Lurie_higher-topos}) whose cofibrations and weak equivalences are the monomorphisms and quasi-isomorphisms, respectively (\cite[3.13]{Beke_sheafifiable-homotopy}), called the \emph{injective model structure} and denoted by $\cpx{\mb{A}}\scr{}{\mr{inj}}$.
Its homotopy category is the unbounded derived category of $\mb{A}$.
Let $\D{\mb{A}}$ denote the stable, locally $\mf{U}$-presentable qcategory underlying $\cpx{\mb{A}}\scr{}{\mr{inj}}$ (\cite[1.3.4.22]{Lurie_higher-algebra}).
\item
If $\mb{A}$ is essentially $\mf{U}$-small and $\mb{A}^{\otimes}$ is a symmetric monoidal structure such that $\prns{-} \otimes \prns{-}$ is exact separately in each variable, then there is a canonical symmetric monoidal structure $\ind{\mb{A}}\scr{\otimes}{}$ such that the Yoneda embedding $\mb{A}^{\otimes} \hookrightarrow \ind{\mb{A}}\scr{\otimes}{}$ is symmetric monoidal and $\prns{-} \otimes \prns{-}$ is exact and $\mf{U}$-cocontinuous separately in each variable on $\ind{\mb{A}}$ (\cite[Expos\'e I, 8.9.8]{SGA4a}, \cite[4.8.1.13]{Lurie_higher-algebra}).
\item
An object $A \in \mb{A}$ is \emph{Noetherian} if each set of subobjects of $A$ contains a maximal element, and $\mb{A}$ is \emph{Noetherian} if each of its objects is Noetherian.
\end{enumerate}
\end{defn}

\begin{lemma}
\label{ab.4}
Let $\mb{A}^{\otimes}$ be a $\mf{U}$-cocomplete symmetric monoidal Abelian category such that $\prns{-} \otimes \prns{-}$ is exact and $\mf{U}$-cocontinuous separately in each variable and such that $\aleph_0$-filtered $\mf{U}$-colimits are exact.
Then $\prns{-} \otimes \prns{-}$ preserves quasi-isomorphisms in $\cpx{\mb{A}}\scr{\otimes}{}$ separately in each variable.
\end{lemma}

\begin{proof}
Let $K \in \cpx{\mb{A}}$.
The additive endofunctor $K \otimes \prns{-}$ of $\cpx{\mb{A}}$ preserves cochain homotopy equivalences and therefore induces a triangulated endofunctor of $\ho{\K{\mb{A}}}$.
It therefore suffices to show that $K \otimes \prns{-}$ preserves acyclic complexes.
Suppose $L \in \cpx{\mb{A}}$ is acyclic.
Write $K \simeq \colim_{r \in \mb{Z}} \trun^{\leq r} K$ and $L \simeq \colim_{s \in \mb{Z}} \trun^{\leq s}L$ as $\aleph_0$-filtered colimits of their truncations with respect to the natural t-structure.
These are bounded above and $\trun^{\leq s}L$ is acyclic for each $s \in \mb{Z}$.
The claim follows from the hypotheses, \ref{ab.1}$(iv)$ and the isomorphisms
$
K \otimes L
\simeq
\prns{\colim_{r \in \mb{Z}}\trun^{\leq r}K} \otimes \prns{\colim_{s \in \mb{Z}} \trun^{\leq s}L}
\simeq
\colim_{r \in \mb{Z}} \colim_{s \in \mb{Z}} \prns{\trun^{\leq r}K \otimes \trun^{\leq s}L}
$.
\end{proof}

\begin{prop}
\label{ab.5}
Let $\mb{A}$ be an essentially $\mf{U}$-small Abelian category.
\begin{enumerate}
\item
The Yoneda embedding $\mb{A} \hookrightarrow \ind{\mb{A}}$ induces an exact functor $\iota : \D{\mb{A}}^{\mr{b}} \to \D{\ind{\mb{A}}}$.
\item
If $\mb{A}$ is Noetherian, then $\iota$ is fully faithful and its essential image is spanned by the cohomologically bounded cochain complexes whose cohomology objects belong to the essential image of the Yoneda embedding $\mb{A} \hookrightarrow \ind{\mb{A}}$.
\item
If $\mb{A}^{\otimes}$ is a symmetric monoidal structure on $\mb{A}$ such that $\prns{-} \otimes \prns{-}$ is exact separately in each variable, then $\iota$ underlies a symmetric monoidal functor.
\end{enumerate}
\end{prop}

\begin{proof}
The natural functor $\cpx{\mb{A}}^{\mr{b}} \to \cpx{\ind{\mb{A}}}$ preserves cochain homotopy equivalences and quasi-isomorphisms, whence a functor $\iota$, which is exact by \cite[1.4.2.14]{Lurie_higher-algebra}.
The induced triangulated functor $\ho{\iota}: \ho{\D{\mb{A}}^{\mr{b}}} \to \ho{\D{\ind{\mb{A}}}}$ is the derived functor of the Yoneda embedding $\mb{A} \hookrightarrow \ind{\mb{A}}$ by construction, so $(ii)$ follows from \cite[2.2]{Huber_calculation-of} and \cite[5.10]{Blumberg-Gepner-Tabuada_universal-characterization}.
Finally, $\D{\ind{\mb{A}}}$ inherits a symmetric monoidal structure from $\cpx{\ind{\mb{A}}}\scr{\otimes}{}$ by \ref{ab.4}, and \ref{thm.4} provides a canonical symmetric monoidal functor $\D{\mb{A}}^{\mr{b}}\scr{\otimes}{} \to \D{\ind{\mb{A}}}\scr{\otimes}{}$, whose underlying functor must, by the universal property of the Verdier quotient $\D{\mb{A}}^{\mr{b}}$, be equivalent to $\iota$.
\end{proof}

\section{Generators}

\begin{motivation}
We continue our study of derived qcategories of Abelian categories, specifically considering the problem of identifying the $\aleph_0$-presentable objects of the unbounded derived qcategory $\D{\ind{\mb{A}}}$ of the ind-completion of a $\mf{U}$-small Abelian category $\mb{A}$.
If, in addition to assuming that $\mb{A}$ is Noetherian, we impose a rather strong finiteness hypothesis, namely, a uniform bound on the homological dimension of the objects of $\mb{A}$, then the essential image of the fully faithful embedding $\D{\mb{A}}^{\mr{b}} \hookrightarrow \D{\ind{\mb{A}}}$ is the full subqcategory $\D{\ind{\mb{A}}}\scr{}{\aleph_0}$ spanned by the $\aleph_0$-presentable objects.
This provides an equivalence $\ind{\D{\mb{A}}^{\mr{b}}} \simeq \D{\ind{\mb{A}}}$.
\end{motivation}

\begin{summary}
We begin by defining the homological dimensions of an object $A \in \mb{A}$ and of $\mb{A}$ itself (\ref{gen.1}), and explaining the relevance of homological dimension to questions of presentability in the derived qcategory (\ref{gen.2}, \ref{gen.3}, \ref{gen.5}).
In Proposition \ref{gen.6}, we construct the desired equivalence $\ind{\D{\mb{A}}^{\mr{b}}} \simeq \D{\ind{\mb{A}}}$ assuming $\mb{A}$ is Noetherian of finite homological dimension.
Corollary \ref{gen.7} shows that, if $\mb{A}^{\otimes}$ is moreover symmetric monoidal and each of its objects is $\otimes$-dualizable, then the $\aleph_0$-presentable objects of $\D{\ind{\mb{A}}}$ are precisely the $\otimes$-dualizable ones.
\end{summary}

\begin{defn}
\label{gen.1}
Let $\mb{A}$ be an Abelian category.
We define the \emph{homological dimension of $A \in \mb{A}$} to be 
$
\on{hdim}\prns{A} 
:=
\sup \brc[1]{n \in \mb{Z}_{\geq0} \mid \exists B \in \mb{A}\ \brk{\ext{A}{B}^n_{\mb{A}} \neq 0}}
\in
\mb{Z}_{\geq0} \cup \brc{\infty}
$
and the \emph{homological dimension of $\mb{A}$} to be $\on{hdim}\prns{\mb{A}} := \sup \brc{\on{hdim}\prns{A} \mid A \in \mb{A}} \in \mb{Z}_{\geq0} \cup \brc{\infty}$.
\end{defn}

\begin{lemma}
\label{gen.2}
Let $\mb{A}$ be an essentially $\mf{U}$-small Abelian category, $A \in \mb{A}$ and $n \in \mb{Z}$.
If $\mb{A}$ is Noetherian and $\on{hdim}\prns{A} < n$, then $\ext{A}{B}^n_{\ind{\mb{A}}} = 0$ for each $B \in \ind{\mb{A}}$.
\end{lemma}

\begin{proof}
Abelian categories are idempotent complete, so the Yoneda embedding $\mb{A} \hookrightarrow \ind{\mb{A}}$ identifies $\mb{A}$ with the full subcategory $\ind{\mb{A}}_{\aleph_0} \subseteq \ind{\mb{A}}$ spanned by the $\aleph_0$-presentable objects.
Indeed, \cite[5.4.2.4]{Lurie_higher-topos} identifies $\ind{\mb{A}}_{\aleph_0}$ with the idempotent completion of $\mb{A}$ and \cite[5.1.4.9]{Lurie_higher-topos} shows that the idempotent completion is essentially unique.
It follows that $\ind{\mb{A}}_{\aleph_0}$ is Abelian and $\ind{\mb{A}}$ is locally $\mf{U}$-coherent in the sense of \cite[\S2.2]{Lowen-Van-den-Bergh_deformation-theory}, i.e., it admits $\mf{U}$-set of $\aleph_0$-presentable generators and the full subcategory spanned by $\aleph_0$-presentable objects is Abelian.
We may therefore apply \cite[6.34]{Lowen-Van-den-Bergh_deformation-theory}, which implies that $\ext{A}{-}^r_{\ind{\mb{A}}}: \ind{\mb{A}} \to \ab$ preserves $\aleph_0$-filtered $\mf{U}$-colimits for each $r \in \mb{Z}_{\geq0}$.

Let $B \in \ind{\mb{A}}$.
By definition of $\ind{\mb{A}}$, $B \simeq \colim_{\gamma \in \Gamma}B_{\gamma}$ for some $\mf{U}$-small $\aleph_0$-filtered diagram $\gamma \mapsto B_{\gamma}: \Gamma \to \mb{A}$.
We now have
\[
\ext{A}{B}^n_{\ind{\mb{A}}}
\simeq
\ext{A}{\colim_{\gamma \in \Gamma}B_{\gamma}}^n_{\ind{\mb{A}}}
\simeq
\colim_{\gamma \in \Gamma} \ext{A}{B_{\gamma}}^n_{\ind{\mb{A}}}
=
0,
\]
since $\ext{A}{B_{\gamma}}^n_{\mb{A}} = \ext{A}{B_{\gamma}}^n_{\ind{\mb{A}}}$ for each $\gamma \in \Gamma$  by \cite[2.2]{Huber_calculation-of}, and $n > \on{hdim}\prns{\mb{A}}$ by hypothesis.
\end{proof}

\begin{lemma}
\label{gen.3}
Let $\mb{A}$ be an essentially $\mf{U}$-small Abelian category, $\iota: \D{\mb{A}}^{\mr{b}} \to \D{\ind{\mb{A}}}$ the natural functor \textup{\ref{ab.5}}, and $A \in \mb{A}$.
If $\mb{A}$ is Noetherian and $\on{hdim}\prns{A} < \infty$, then $\iota A$ is $\aleph_0$-presentable.
\end{lemma}

\begin{proof}
We identify $A$ with $\iota A$.
By \cite[1.4.4.1(3)]{Lurie_higher-algebra}, $A$ is $\aleph_0$-presentable in $\D{\ind{\mb{A}}}$ if and only if $\pi_0\map{A}{-}_{\D{\ind{\mb{A}}}}: \D{\ind{\mb{A}}} \to \D{\ab}$ preserves $\mf{U}$-coproducts.
Let $\brc{K_{\gamma}}_{\gamma \in \Gamma}$ be a $\mf{U}$-set of objects of $\D{\ind{\mb{A}}}$ and let $K := \coprod_{\gamma \in \Gamma} K_{\gamma}$.
By \ref{gen.2}, the hypercohomology spectral sequences
\begin{align}
\label{gen.5.1}
\mr{E}^{r,s}_2
= 
\ext{A}{\h^rK}^s_{\ind{\mb{A}}}
&\Longrightarrow
\pi_0\map{A}{K\brk{r+s}}_{\D{\ind{\mb{A}}}}
\\
\label{gen.5.2}
\prns{\mr{E}_{\gamma}}^{r,s}_2
=
\ext{A}{\h^rK_{\gamma}}^s_{\ind{\mb{A}}}
&\Longrightarrow
\pi_0\map{A}{K_{\gamma}\brk{r+s}}_{\D{\ind{\mb{A}}}}
\end{align}
are uniformly bounded and hence strongly convergent for each $\gamma \in \Gamma$.
As $\aleph_0$-filtered $\mf{U}$-colimits are exact in the $\mf{U}$-Grothendieck Abelian category $\ind{\mb{A}}$, they are t-exact in $\D{\ind{\mb{A}}}$, which, when combined with \cite[6.34]{Lowen-Van-den-Bergh_deformation-theory}, gives an isomorphism $\bigoplus_{\gamma \in \Gamma} \prns{\mr{E}_{\gamma}}^{r,s}_2 \simeq \mr{E}^{r,s}_2$.
It follows that the abutment of \eqref{gen.5.1} is the direct sum of the abutments of the spectral sequences $\eqref{gen.5.2}_{\gamma}$ and the canonical morphism
\[
\bigoplus_{\gamma \in \Gamma} \pi_0\map{A}{K_{\gamma}}_{\D{\ind{\mb{A}}}}
\to
\pi_0\map{A}{K}_{\D{\ind{\mb{A}}}}
\]
is an isomorphism, as required.
\end{proof}

\begin{defn}
\label{gen.4}
Let $\mc{C}$ be a stable qcategory and $\mc{G} \subseteq \mc{C}$ a full subqcategory.
If $X \in \mc{C}$ is a zero object whenever $\map{G\brk{n}}{X}_{\mc{C}}$ is contractible for each $G \in \mc{G}$ and each $n \in \mb{Z}$, then we say that $\mc{G}$ \emph{generates $\mc{C}$ as a stable qcategory} or we say that $\mc{G}$ is a \emph{generating family in $\mc{C}$}.

This definition---a direct translation of the notion of a generating family in a triangulated category as developed in \cite{Neeman_triangulated-categories}---is not equivalent to the one introduced before Corollary 1.4.4.2 in \cite{Lurie_higher-algebra}.
Assuming $\mc{C}$ admits sufficiently large coproducts, a generating family $\mc{G}$ in our sense induces a generator $G := \coprod_{\mc{G} \times \mb{Z}}G\brk{n}$ in J.~Lurie's sense.
The classical definition is better adapted to a discussion of presentability: if $\mc{C}$ is locally $\mf{U}$-presentable and each $G \in \mc{G}$ is $\aleph_0$-presentable, then $X$ is almost never $\aleph_0$-presentable, even if $\mc{G}$ is spanned by a single object.
\end{defn}

\begin{prop}
\label{gen.5}
Let $\mb{A}$ be an essentially $\mf{U}$-small Abelian category and $\mc{G} := \brc{ A \brk{n} \mid A \in \mb{A},\ n \in \mb{Z}}$, where $A\brk{n}$ denotes the cochain complex in $\ind{\mb{A}}$ given by $A$ and concentrated in degree $-n$.
If $\mb{A}$ is Noetherian and $\on{hdim}\prns{\mb{A}} < \infty$, then $\mc{G}$ is a generating family of $\aleph_0$-presentable objects in $\D{\ind{\mb{A}}}$.
\end{prop}

\begin{proof}
The $\aleph_0$-presentability assertion follows from \ref{gen.3}.
That $\mc{G}$ is a generating family follows from the observation that, if $\h^nK \neq 0$ for some $n \in \mb{Z}$ and $\d^n: K^n \to K^{n+1}$ denotes the differential of $K$, then the canonical morphism of complexes $\ker\prns{\d^n} \to K$ is not equivalent to the zero morphism in $\D{\ind{\mb{A}}}$, where $\ker\prns{\d^n}$ is seen as a complex concentrated in degree $n$.
Up to translation, $\ker\prns{\d^n} \in \ind{\mb{A}}$ and $\ker\prns{\d^n}$ therefore admits a morphism from an element of $\mc{G}$ which is not equivalent to the zero morphism in $\D{\ind{\mb{A}}}$ by definition of $\ind{\mb{A}}$.
\end{proof}

\begin{prop}
\label{gen.6}
Let $\mb{A}$ be an essentially $\mf{U}$-small Noetherian Abelian category such that $\on{hdim}\prns{\mb{A}} < \infty$, and $\iota: \D{\mb{A}}^{\mr{b}} \hookrightarrow \D{\ind{\mb{A}}}$ the natural functor \textup{(\ref{ab.5})}.
For each $K \in \D{\ind{\mb{A}}}$, the following are equivalent:
\begin{enumerate}
\item
$K$ is equivalent to an object of the smallest idempotent-complete stable subqcategory of $\D{\ind{\mb{A}}}$ containing the generating family $\mc{G}$ of \textup{\ref{gen.5}};
\item
$K$ belongs to $\D{\mb{A}}^{\mr{b}}$;
\item
$K$ is $\aleph_0$-presentable.
\end{enumerate}
In particular, $\iota$ induces an equivalence $\ind{\D{\mb{A}}^{\mr{b}}} \simeq \D{\ind{\mb{A}}}$ and $\D{\ind{\mb{A}}}$ is locally $\aleph_0$-presentable.
\end{prop}

\begin{proof}
By \ref{ab.5}, $\iota$ is exact and fully faithful. 
This justifies the abuse in $(ii)$ of identifying $\D{\mb{A}}^{\mr{b}}$ with its essential image under $\iota$.
Since $\D{\mb{A}}^{\mr{b}}$ is an idempotent-complete stable subqcategory of $\D{\ind{\mb{A}}}$, $(i)$ implies $(ii)$.

By \ref{gen.3}, the elements of $\mc{G}$ are $\aleph_0$-presentable.
In terms of the natural t-structure on $\D{\mb{A}}^{\mr{b}}$, this means that each object of $\D{\ind{\mb{A}}}$ whose cohomology is concentrated in a single degree and given by an object $A \in \mb{A}$ is $\aleph_0$-presentable.
Using the fiber sequences associated with the truncation functors for the natural t-structure and the fact that each finite colimit of $\aleph_0$-presentable objects is $\aleph_0$-presentable, this proves by induction that each $K \in \D{\ind{\mb{A}}}$ whose cohomology objects are concentrated in finitely many degrees and belong to $\mb{A}$ is $\aleph_0$-presentable.
Thus, $(ii)$ implies $(iii)$ by \ref{ab.5}$(ii)$.

It remains to prove $(iii)$ implies $(i)$.
Define a chain $\mc{C}\prns{0} \subseteq \mc{C}\prns{1} \subseteq \cdots$ of full subqcategories of $\D{\ind{\mb{A}}}$ inductively as follows: 
$\mc{C}\prns{0}$ is the full subqcategory spanned by $\mc{G}$ of \ref{gen.5}; 
and if $\mc{C}\prns{r}$ has been defined for some $r \in \mb{Z}$, define $\mc{C}\prns{r+1}$ to be the full subqcategory spanned by all objects arising as colimits of finite diagrams in $\mc{C}\prns{r}$.
Letting $\mc{C}\prns{\bs{\omega}} := \bigcup_{r \in \mb{Z}_{\geq 0}} \mc{C}\prns{r}$, the argument of \cite[1.4.4.2]{Lurie_higher-algebra} now shows that $\ind{\mc{C}\prns{\bs{\omega}}} \simeq \D{\ind{\mb{A}}}$.
It follows then from \cite[5.4.2.4]{Lurie_higher-topos} that $\D{\ind{\mb{A}}}\scr{}{\aleph_0}$ is an idempotent completion of $\mc{C}\prns{\bs{\omega}}$.
Note that this argument relies on the fact that $\mc{G}$ is a generating family of $\aleph_0$-presentable objects (\ref{gen.5}).
In any event, the idempotent completion of $\mc{C}\prns{\bs{\omega}}$ is precisely the smallest idempotent-complete stable subqcategory of $\D{\ind{\mb{A}}}$ containing $\mc{G}$ and the claim follows.
The equivalence $\ind{\mc{C}\prns{\bs{\omega}}} \simeq \D{\ind{\mb{A}}}$ also proves the final assertion.
\end{proof}

\begin{cor}
\label{gen.7}
Let $\mb{A}^{\otimes}$ be an essentially $\mf{U}$-small $\KK$-linear symmetric monoidal Abelian category.
Suppose that $\mb{A}$ is Noetherian, $\on{hdim}\prns{\mb{A}} < \infty$ and each object of $\mb{A}$ is $\otimes$-dualizable \textup{(\cite[4.6.1.12]{Lurie_higher-algebra})}.
The fully faithful functor $\iota: \D{\mb{A}}^{\mr{b}} \hookrightarrow \D{\ind{\mb{A}}}$ of \textup{\ref{ab.5}} underlies a symmetric monoidal functor and, for each $K \in \D{\ind{\mb{A}}}$, the equivalent conditions $(i)$, $(ii)$ and $(iii)$ of \textup{\ref{gen.6}} are furthermore equivalent to the following:
\begin{enumerate}
\item[$(iv)$]
$K$ is $\otimes$-dualizable.
\end{enumerate}
In particular, $\iota^{\otimes}$ induces an equivalence $\ind{\D{\mb{A}}^{\mr{b}}}\scr{\otimes}{} \simeq \D{\ind{\mb{A}}}\scr{\otimes}{}$.
\end{cor}

\begin{proof}
The hypothesis that each $A \in \mb{A}$ be $\otimes$-dualizable implies that the tensor product in $\mb{A}^{\otimes}$ is exact separately in each variable: for each $A \in \mb{A}$, $A \otimes \prns{-}$ is both left and right adjoint functor and therefore preserves all limits and colimits that exist in $\mb{A}$.
By \ref{ab.5}$(iii)$, we have the desired symmetric monoidal functor $\iota^{\otimes}: \D{\mb{A}}^{\mr{b}}\scr{\otimes}{} \hookrightarrow \D{\ind{\mb{A}}}\scr{\otimes}{}$.
It follows from \cite[1.3.4.5, 4.1.3.4]{Lurie_higher-algebra} and \ref{ab.2} that we have symmetric monoidal functors 
\[
\mb{A}^{\otimes} 
\to 
\cpx{\mb{A}}^{\mr{b}}\scr{\otimes}{} 
\to 
\K{\mb{A}}^{\mr{b}}\scr{\otimes}{} 
\to 
\D{\mb{A}}^{\mr{b}}\scr{\otimes}{}
\]
whose composite sends $A$ to $A\brk{0}$, and thus that $\mc{G}$ consists of $\otimes$-dualizable objects.
The full subqcategory spanned by the $\otimes$-dualizable objects is stable under finite limits and retracts, so $(i)$ implies $(iv)$.

We claim that $(iv)$ implies $(iii)$.
If $K$ is $\otimes$-dualizable, $\brc{L_{\gamma}}_{\gamma \in \Gamma}$ is a $\mf{U}$-family of objects of $\D{\ind{\mb{A}}}$ and $L := \coprod_{\gamma \in \Gamma}L_{\gamma}$, then 
\begin{align*}
\pi_0\map{K}{L}_{\D{\ind{\mb{A}}}}
&\simeq
\pi_0\map{\1{\mb{T}}}{K^{\vee} \otimes L}_{\D{\ind{\mb{A}}}}
&&
\text{adjunction}
\\
&\simeq
\pi_0\map[1]{\1{\mb{A}}}{\coprod_{\gamma \in \Gamma}\prns{ K^{\vee} \otimes L_{\gamma}}}_{\D{\ind{\mb{A}}}}
\\
&\simeq
\bigoplus_{\gamma \in \Gamma} \pi_0\map{\1{\mb{A}}}{K^{\vee} \otimes L_{\gamma}}_{\D{\ind{\mb{A}}}}
&&
\textup{\ref{gen.5}}
\\
&\simeq
\bigoplus_{\gamma \in \Gamma} \pi_0\map{K}{L_{\gamma}}_{\D{\ind{\mb{A}}}}
\end{align*}
and $K$ is therefore $\aleph_0$-presentable.
\end{proof}

\allsectionsfont{\centering\fontseries{m}\scshape}
\bibliographystyle{alpha}
\bibliography{all}
\end{document}